\documentclass[10pt]{amsart}
\usepackage{enumerate,amsmath,amssymb,latexsym,
amsfonts, amsthm, amscd, MnSymbol}

\theoremstyle{plain}

\newtheorem{theorem}{Theorem}

\numberwithin{equation}{section}

\newcommand{\ra}{\rightarrow}

\addtocounter{section}{-1}

\begin{document}

\title {Finite N\"orlund summation methods}

\date{}

\author[P.L. Robinson]{P.L. Robinson}

\address{Department of Mathematics \\ University of Florida \\ Gainesville FL 32611  USA }

\email[]{paulr@ufl.edu}

 \keywords{}

\begin{abstract}

We draw attention to simplifications in the theory of a N\"orlund summation method $(N, p)$ that arise when the series $\sum_{n \geqslant 0} p_n$ is convergent.

\end{abstract}

\maketitle

\medbreak 

Let the real sequence $(p_n : n \geqslant 0)$ satisfy $p_0 > 0$ and $p_n \geqslant 0$ for $n > 0$. Let $P_n = p_0 + \cdots + p_n$ when $n \geqslant 0$ and let 
$$P = \lim_{ n \ra \infty} P_n = \sum_{n \geqslant 0} p_n \leqslant \infty.$$
The corresponding N\"orlund process $(N, p)$ associates to each sequence $s = (s_n : n \geqslant 0)$ the sequence $N^{(p)}(s) = t = (t_m : m \geqslant 0)$ given by 
$$t_m = N^{(p)}_m(s) = \frac{p_0 s_m + \cdots + p_m s_0}{p_0 + \cdots + p_m} = \frac{1}{P_m} \sum_{n = 0}^m p_{m - n} s_n.$$
Precisely when the sequence $(t_m: m \geqslant 0)$ converges to $\sigma$ in the ordinary sense, we say that the sequence $(s_n : n \geqslant 0)$ is $(N, p)$-convergent to $\sigma$ and write 
$$s_n \xrightarrow{(N, p)} \sigma$$
or, for typographical reasons, $s_n \ra \sigma \; (N, p)$; in case $(s_n : n \geqslant 0)$ is the sequence of partial sums of the series $\sum_{n \geqslant 0} a_n$ we say that this series is $(N, p)$-summable with sum $\sigma$ and write 
$$\sum_{n \geqslant 0} a_n = \sigma \; (N, p).$$
The N\"orlund process $(N, p)$ is one of many summation methods for assigning sums to ordinarily divergent series. As our basic reference, we take the classic treatise {\it `Divergent Series}' by Hardy [1]: general theorems regarding summation methods are covered in Chapter III;  N\"orlund methods themselves open Chapter IV. 

\medbreak 

Our primary concern is to highlight certain simplifications that take place in the theory of N\"orlund methods when attention is limited to those that are finite. Here, we say that the N\"orlund method $(N, p)$ is {\it finite} precisely when $P < \infty$; that is, when the series $\sum_{n \geqslant 0} p_n$ is convergent. With all due respect, rather than refer to $(N, p)$ as a N\"orlund process or a N\"orlund method, we may refer to it briefly as a `N\"orlund'. 

\medbreak 

The first simplification has to do with regularity. Quite generally, a summation method is said to be {\it regular} precisely when it assigns to each ordinarily convergent series its ordinary sum. Theorem 16 in [1] establishes that the general N\"orlund $(N, p)$ is regular precisely when the sequence of quotients $(p_n/P_n : n \geqslant 0)$ converges to zero. For a finite N\"orlund, this simplifies as follows. 

\medbreak 

\begin{theorem} \label{reg} 
Each finite N\"orlund method is regular. 
\end{theorem} 

\begin{proof} 
Let the series $\sum_{n \geqslant 0} p_n$ be convergent: its terms $p_n$ converge to $0$ and its partial sums $P_n$ converge to $P > 0$; consequently, $p_n/P_n \rightarrow 0$. 
\end{proof} 

More interesting simplifications have to do with comparisons between summation methods. We say that one summation method {\it includes} a second summation method precisely when each series that is summable by the second is summable by the first (to the same sum); we say that two summation methods are {\it equivalent} exactly when each includes the other. Precise necessary and sufficient conditions for inclusion and equivalence between regular N\"orlunds were determined by Marcel Riesz; an extract from his letter to Hardy announcing these results was published as [2]. For convenience, throughout the following discussion we shall consistently refer to [1], where these results appear as Theorem 19 and Theorem 21.  

\medbreak 

To prepare for the Riesz theorems and their simplifications, let $(N, p)$ and $(N, q)$ be regular N\"orlunds, with $p_0 + \cdots + p_n = P_n \ra P$ as $n \ra \infty$ and $q_0 + \cdots + q_n = Q_n \ra Q$ as $n \ra \infty$. Associated to this N\"orlund pair are comparison sequences $(k_n : n \geqslant 0)$ and $(l_n : n \geqslant 0)$ uniquely determined by recursively solving the convolution systems 
$$q_n = k_0 p_n + \cdots + k_n p_0$$
$$p_n = l_0 q_n + \cdots + l_n q_0$$
since $p_0$ and $q_0$ are nonzero. By summation, it follows that also 
$$Q_n = k_0 P_n + \cdots + k_n P_0$$
$$P_n = l_0 Q_n + \cdots + l_n Q_0.$$
The N\"orlund coefficients may be assembled to define power series 
$$p(x) = \sum_{n \geqslant 0} p_n x^n$$
$$q(x) = \sum_{n \geqslant 0} q_n x^n$$
convergent for $|x| < 1$ and nonzero when $|x|$ is small. The comparison coefficients likewise assemble to define mutually reciprocal power series 
$$k(x) = \sum_{n \geqslant 0} k_n x^n$$
$$l(x) = \sum_{n \geqslant 0} l_n x^n$$
satisfying 
$$q(x) = k(x) p(x)$$
$$p(x) = l(x) q(x)$$
and converging when $|x|$ is small. 

\medbreak 

In these terms, Theorem 19 in [1] establishes that $(N, q)$ includes $(N, p)$ if and only if each of the following conditions is satisfied: \par 
(i) there exists $H > 0$ such that for each $n \geqslant 0$ 
$$|k_0| P_n + \cdots + |k_n| P_0 \leqslant H Q_n;$$\par 
(ii) the sequence of quotients $(k_n/Q_n: n \geqslant 0)$ converges to zero. \\ 

\medbreak 

In order to exhibit the simplification that comes from assuming that $(N, p)$ and $(N, q)$ are both finite, we introduce two notational conveniences. First, to indicate that $(N, q)$ includes $(N, p)$ we shall write $(N, p) \rightsquigarrow (N, q)$; this symbolizes the requirement that for any sequence, $(N, p)$-convergence implies $(N, q)$-convergence. Next, we define 
$$[q:p] = \sum_{n \geqslant 0} |k_n| \leqslant \infty$$
so that likewise 
$$[p:q] = \sum_{n \geqslant 0} |l_n|.$$

\medbreak 

Theorem 19 in [1] now simplifies as follows. 

\begin{theorem} \label{inc}
If $(N, p)$ and $(N, q)$ are finite N\"orlund methods, then 
$$(N, p) \rightsquigarrow (N, q) \Leftrightarrow [q:p] < \infty.$$
\end{theorem} 

\begin{proof} 
In the forward direction, let $(N, p) \rightsquigarrow (N, q)$: [1] Theorem 19 yields 
$$|k_0| P_n + \cdots + |k_n| P_0 \leqslant H Q_n$$
whence
$$ |k_0| + \cdots + |k_n| \leqslant H Q /P_0$$
for each $n \geqslant 0$; so  
$$[q : p] = \sum_{n \geqslant 0} |k_n| \leqslant H Q / P_0 < \infty.$$
In the reverse direction, let $[q : p] < \infty$. From $P_n/Q_n \ra P/Q$ we deduce that $P_n / Q_n$ is bounded; say $P_n /Q_n \leqslant J$. Now 
$$|k_n| P_0 + \cdots + |k_0| P_n \leqslant (|k_n| + \cdots + |k_0|) P_n \leqslant \sum_{\nu = 0}^{n} |k_{\nu}| J Q_n \leqslant [q:p] J Q_n.$$
The conditions of [1] Theorem 19 are thus satisfied: the first with $H = [q:p] J$; the second because $k_n \ra 0$. 
\end{proof} 

\medbreak 

Theorem 21 in [1] appears with no essential change to its conclusion. 

\medbreak 

\begin{theorem} \label{equ}
The finite N\"orlund methods $(N, p)$ and $(N, q)$ are equivalent if and only if both $[q:p] < \infty$ and $[p:q] < \infty.$
\end{theorem} 

\begin{proof} 
Aside from its referring only to finite N\"orlunds, the {\it statement} of this second Riesz theorem has undergone no change of substance; in the finite setting, its {\it proof} has simplified to the point that it now follows immediately from the first (Theorem \ref{inc}). 
\end{proof} 

\medbreak 

We remark that when regular N\"orlunds are considered, the condition $[q:p] < \infty$ is neither necessary nor sufficient for $(N, q)$ to include $(N, p)$. This is suggested by the inextricable entanglement of the conditions $[q:p] < \infty$ and $[p:q] < \infty$ in the proof of Theorem 21 in [1] and is manifestly clear from the following example. 

\medbreak 

{\bf Example 0}. 
Let $u_0 = 1$ and let $u_n = 0$ whenever $n > 0$; application of the method $(N, u)$ leaves each sequence unchanged, and $(N, u)$-convergence is ordinary convergence. Let $c_n = 1$ for every $n \geqslant 0$; application of the method $(N, c)$ converts any sequence to its sequence of arithmetic means, and $(N, c)$-convergence is Ces\`aro convergence. Here, $u(x) = 1$ and $c(x) = \sum_{n \geqslant 0} x^n$ so that 
$c(x)/u(x) = \sum_{n \geqslant 0} x^n$ and $u(x)/c(x) = 1 - x.$ On the one hand, $[c:u] = \infty$ but the inclusion $(N, u) \rightsquigarrow (N, c)$ holds, as it amounts to regularity of the Ces\`aro method $(N, c)$. On the other hand, $[u:c] = 2 < \infty$ but the inclusion $(N, c) \rightsquigarrow (N, u)$ fails, since the ordinarily divergent sequence $(1, 0, 1, 0, \dots)$ is Cesa\`ro convergent (to $1/2$). 

\medbreak 

Further simplifications regarding finite N\"orlund methods have to do with triviality. Here, we say that a summation method is {\it trivial} precisely when it is equivalent to ordinary summation; that is, precisely when it is equivalent to $(N, u)$ in the notation of Example 0. 

\medbreak 

Hardy [1] quotes as Theorem 22 a result of Kaluza and Szeg\"o to the following effect: let $p_0 = 1$, let $p_n > 0$ when $n > 0$, and let the power series $\sum_{n \geqslant 0} p_n x^n$ converge when $|x| < 1$; if the coefficients satisfy the condition $p_{n + 1} p_{n - 1} \geqslant p_n^2$ whenever $n > 0$ then 
$$\frac{1}{p(x)} = \sum_{n \geqslant 0} k_n x^n$$
where $k_0 = 1$, where $k_n \leqslant 0$ for $n > 0$, and where $\sum_{n > 0} k_n \geqslant -1$. Hardy puts this result to use in proving another inclusion theorem: namely, his [1] Theorem 23 establishes that if the regular N\"orlund $(N, p)$ has the foregoing properties and if $(N, q)$ is a regular N\"orlund with strictly positive coefficients such that $p_{n + 1}/p_n \leqslant q_{n + 1}/q_n$ eventually, then $(N, p) \rightsquigarrow (N, q)$. Now, Hardy is here interested primarily in cases in which $P_n$ tends slowly to infinity; with good reason. In fact, if the N\"orlund $(N, p)$ presently under consideration were finite then it would be trivial: indeed, $k_0 = 1$ and the inequality $\sum_{n > 0} k_n \geqslant - 1$ above imply that 
$$[u : p] = \sum_{n \geqslant 0} |k_n| \leqslant 2$$
which with $[p:u] = P < \infty$ renders $(N, p)$ equivalent to $(N, u)$ of Example 0 according to Theorem \ref{equ}. Consequently, [1] Theorem 23 trivializes in such a case, for then $(N, p) \rightsquigarrow (N, q)$ when $(N, q)$ is any regular N\"orlund process whatever, by the definition of regularity. 

\medbreak 

We conclude our account with some further examples. In each, the symbol $p$ plays various r\^oles, distinguishable by context. 

\medbreak 

{\bf Example 1}. 
Fix a positive real number $p$ and let $p_n = p^n / n!$ whenever $n \geqslant 0$. The `Poisson' N\"orlund $(N, p)$ is always trivial: $p(x) = e^{p x}$ and $1/p(x) = e^{-px}$ so that $[p:u] = e^p = [u:p]$. 

\medbreak 

{\bf Example 2}. 
Again fix $p > 0$ and let $p_n = p^n$ whenever $n \geqslant 0$. The `geometric' N\"orlund $(N, p)$ is trivial when $p < 1$: explicitly, 
$$p(x) = \sum_{ n \geqslant 0} p^n x^n = (1 - p x)^{-1}$$
so that $[u:p] = 1 + p$ and of course $[p:u] = (1 - p)^{-1}$. When $p = 1$ the `geometric' N\"orlund reduces to the Ces\`aro method. When $p > 1$, the `geometric' N\"orlund $(N, p)$ is no longer regular: instead, it belongs to the class $\mathfrak{T}_c^*$ defined in [1] Chapter III; thus, it transforms each bounded sequence to a convergent one. 

\medbreak 

{\bf Example 3}. 
More generally, fix $p > 0$ but now fix also a positive integer $k$ and let 
$$p_n = {{n + k - 1} \choose {k - 1}} p^n$$
so that if $| p x | < 1$ then 
$$p(x) = \sum_{n \geqslant 0}{{n + k - 1} \choose {k - 1}}(p x)^n = (1 - p x)^{- k}.$$
Here, if $p < 1$ then this `negative binomial' N\"orlund $(N, p)$ is again trivial, for $[p:u] = (1 - p)^{-k}$ and $[u:p] = (1 + p)^k$. If $p = 1$ then $(N, p)$ is the standard Ces\`aro summation method $(C, k)$.

\medbreak 

{\bf Example 4}. 
Fix $s$ and let $p_n = (n + 1)^{-s}$. If $s > 1$ then $(N, p)$ is finite with $[p:u] = \zeta(s)$ and is indeed trivial, because if $s > 0$ then $p_{n + 1} p_{n - 1} > p_n^2$ and  [1] Theorem 22 applies. The case $s = 1$ effectively reproduces the `harmonic' summation method of Riesz [2]. 

\medbreak 

{\bf Example 5}. 
The N\"orlund $(N, p)$ will of course be finite when its coefficients vanish beyond some index, whereupon the power series $p(x)$ reduces to a polynomial in $x$. Here, the Enestr\"om-Kakeya theorem addresses a special case: it guarantees that if $p_0 > p_1 > \cdots > p_N > 0$ then the polynomial $p(x) = p_0 + p_1 x + \cdots + p_N x^N$ has no zeros in the closed unit disc, whence the power series expansion
$$\frac{1}{p(x)} = \sum_{n \geqslant 0} k_n x^n$$
of its reciprocal converges in a neighbourhood of the closed unit disc; in particular, $\sum_{n \geqslant 0} k_n$ is absolutely convergent and $[u:p] < \infty$. In short, a N\"orlund $(N, p)$ for which $p_0 > p_1 > \cdots > p_N > 0$ and $p_n = 0$ whenever $n > N$ is necessarily trivial. To take a specific example, let $p_0 = 1$, let $p_1 = p > 0$, and let $p_n = 0$ whenever $n > 1$: when $p < 1$ this N\"orlund is trivial; when $p = 1$ it reduces to the (nontrivial) Hutton summation method $(Hu, 1)$ mentioned in the notes to Chapter 1 of [1]. 

\medbreak

\bigbreak

\begin{center} 
{\small R}{\footnotesize EFERENCES}
\end{center} 

\medbreak 

[1] G.H. Hardy, {\it Divergent Series}, Clarendon Press, Oxford (1949). 

\medbreak 

[2] M. Riesz, {\it Sur l'\'equivalence de certaines m\'ethodes de sommation}, Proceedings of the London Mathematical Society (2) {\bf 22} 412-419 (1924). 

\medbreak

\end{document}